\documentclass[a4paper, 12pt]{article}

\usepackage{amssymb}
\usepackage{bm}
\usepackage{bbm}
\usepackage{textcomp}
\usepackage{mathrsfs}
\usepackage{amsmath}
\usepackage{tcolorbox}
\usepackage{amsfonts}
\usepackage[T1]{fontenc}
\usepackage[latin9]{inputenc}
\usepackage{amsthm}
\usepackage{graphicx}
\usepackage{amsmath}

\usepackage{amsthm}
\usepackage{array}
\usepackage{cases}
\usepackage{enumerate}
\usepackage{clrscode}
\usepackage{listings}
\usepackage[ruled,vlined,english,linesnumbered]{algorithm2e}
\usepackage{url}
\usepackage{xcolor}  
\usepackage{algpseudocode}  

\usepackage{enumerate,enumitem,todonotes}
\makeatletter
\def\th@plain{%
  \itshape 
}
\makeatother

\makeatletter
\renewenvironment{proof}[1][\proofname]{\par
  \pushQED{\qed}%
  \normalfont \topsep6\p@\@plus6\p@\relax
  \trivlist
  \item[\hskip\labelsep
        \bfseries
    #1\@addpunct{.}]\ignorespaces
}{%
  \popQED\endtrivlist\@endpefalse
}
\makeatother

\numberwithin{equation}{section}

\newtheorem{thm}{Theorem}[section]

\newtheorem{claim}{Claim}
\newtheorem{lem}[thm]{Lemma}

\numberwithin{equation}{section}

\usepackage[varg]{txfonts}
\usepackage{graphicx, epsfig, subfigure}
\usepackage{enumerate} 
\usepackage[square, numbers, sort&compress]{natbib} 
\ifx\pdfoutput\undefined
 \usepackage[dvipdfm,%
  pdfstartview=FitH, 
  bookmarks=true,%
  bookmarksnumbered=true, 
  bookmarksopen=true, 
  plainpages=false,%
  pdfpagelabels,%
  colorlinks=true, 
  linkcolor=blue, 
  citecolor=blue,%
  urlcolor=black,
  pdfborder=001]{hyperref}
  \AtBeginDvi{}  
\else
 \usepackage[pdftex,%
  pdfstartview=FitH, 
  bookmarks=true,%
  bookmarksnumbered=true, 
  bookmarksopen=true, 
  plainpages=false,%
  pdfpagelabels,%
  colorlinks=true, 
  linkcolor=blue, 
  citecolor=blue,%
  urlcolor=black,
  pdfborder=001]{hyperref}
\fi

\numberwithin{equation}{section}

\title{\LARGE Equitable coloring of graphs beyond planarity}
\author{Weichan Liu\\
{\small  School of Mathematics, Shandong University, Jinan, 250100, China}\\
   {\small wcliu@sdu.edu.cn}
}

\begin{document} \baselineskip 0.6cm

\maketitle

\begin{abstract}\baselineskip 0.6cm
An equitable coloring of a graph is a proper coloring where the sizes of any two different color classes do not differ by more than one. A graph is IC-planar if it can be drawn in the plane so that
no two crossed edges have a common endpoint, and is NIC-planar graphs if it can be embedded in the plane in such a way that no two pairs of crossed edges share two endpoints. Zhang proved that every IC-planar graph with maximum degree $\Delta\geq 12$ and every NIC-planar graph with maximum degree $\Delta\geq 13$ have equitable $\Delta$-colorings. In this paper, we reduce the threshold from 12 to 10 for IC-planar graphs and from 13 to 11 for NIC-planar graphs.

\vspace{1em}

\end{abstract}

\section{Introduction}

An equitable $r$-coloring of a graph $G$ is defined as a partition of the vertex set $V(G)$ into $r$ independent subsets $V_1, \ldots, V_r$, such that the difference in cardinalities between any two subsets is at most one. This conceptualization can be visualized as assigning a distinct color $i$ to all vertices within $V_i$, thereby transforming each subset into a color class within the equitable $r$-coloring framework.
 
The concept of equitable graph coloring was formally introduced by Meyer \cite{Meyer1973} in 1973. Its inception was motivated by Tucker's work \cite{doi:10.1137/1015072}, which modeled the problem of optimizing garbage collection routes using graph theory. It is straightforward to demonstrate, using a greedy algorithm, that any graph $G$ admits a proper coloring utilizing at most $\Delta(G) + 1$ colors. This observation inspired Erd\H{o}s \cite{erdos1959} to conjecture in 1959 that every graph $G$ possesses an equitable $r$-coloring for every integer $r \geq \Delta(G) + 1$. This conjecture was subsequently proven by Hajnal and Szemer\'edi \cite{HS1970}. In 2008, Kierstead and Kostochka \cite{zbMATH05302004} provided a simplified proof of the Hajnal-Szemer\'edi theorem, and in 2010, Kierstead, Kostochka, Mydlarz, and Szemer\'edi \cite{zbMATH05881225} presented an algorithmic proof, establishing that such an equitable coloring can be computed in $O(|G|^2 r)$ time.
 
Brooks' theorem, a cornerstone result in graph theory, asserts that if $G$ is neither an odd cycle nor a complete graph, then $\Delta(G)$ colors are sufficient for a proper vertex coloring. Inspired by this, Meyer \cite{Meyer1973} posed the question of whether an equitable analogue of Brooks' theorem holds.

\vspace{2mm}
\noindent \textbf{Equitable Coloring Conjecture:}
\textit{If $G$ is a connected graph that is neither an odd cycle nor a complete graph, then it admits an equitable $\Delta(G)$-coloring.}
\vspace{2mm}

It is important to note that equitable colorings do not necessarily exhibit monotonicity. For example, the complete bipartite graph $K_{3,3}$ can be equitably 2-colored but does not permit any equitable 3-coloring. This observation leads to an intriguing problem: determining the smallest integer $r'$ such that a graph $G$ has an equitable $r$-coloring for every $r \geq r'$. In 1994, Chen, Lih, and Wu \cite{CHEN1994443}  proposed the following conjecture in this regard.
 
\vspace{2mm}
\noindent \textbf{Chen-Lih-Wu Conjecture:}
\textit{If $G$ is a connected graph that is neither an odd cycle, a complete graph, nor a complete bipartite graph of the form $K_{2m+1,2m+1}$, then it has an equitable $r$-coloring for every $r \geq \Delta(G)$.}
\vspace{2mm} 

According to the Hajnal-Szemer\'edi theorem, the Equitable Coloring Conjecture and the Chen-Lih-Wu Conjecture are equivalent. 
To date, these conjectures have been confirmed for graphs $G$ satisfying several conditions on their maximum degree $\Delta(G)$: specifically, when $\Delta(G) \geq |G|/2$ \cite{CHEN1994443}, when $(|G|+1)/3 \leq \Delta(G) < |G|/2$ \cite{chen2014equitable},
when $\Delta(G) \geq |G|/4$ \cite{zbMATH06699287}, when $\Delta(G) = 3$ \cite{CHEN1994443}, or when $\Delta(G) = 4$ \cite{zbMATH06081427}.

In the context of planar graphs $G$ (i.e.,\,graphs that can be embedded in the plane without any edge crossings), 
Zhang and Yap \cite{zbMATH01308943} established that if $\Delta(G) \geq 13$, then $G$ possesses an equitable $\Delta(G)$-coloring. 
Nakprasit \cite{NAKPRASIT20121019} subsequently extended this result to cover the range $9 \leq \Delta(G) \leq 12$. 
Most recently, Kostochka, Lin, and Xiang \cite{zbMATH07840571} further refined the bound by showing that $\Delta(G) \geq 8$ is sufficient.

For non-planar graphs, the exploration of ``Graph Drawing Beyond Planarity'' has emerged as a rapidly advancing research domain, focusing on categorizing and analyzing their geometric representations, with a particular emphasis on identifying and examining prohibited crossing patterns.
Currently, 1-planar graphs have garnered considerable attention as a prominent subclass within the broader category of graphs extending beyond planarity, as detailed in a survey paper by Kobourov, Liotta, and Montecchiani \cite{zbMATH06782829}.
In this context, a graph is defined as \textit{1-planar} if it can be drawn in the plane such that each edge intersects no more than one other edge.

The equitable coloring of 1-planar graphs was initially explored by Zhang \cite{zbMATH06602931}. 
By examining the class $\mathcal{G}$ of graphs $G$ where each subgraph $H$ of $G$ satisfies $e(H) \leq 4v(H) - 8$, which encompasses a broader range of graphs than 1-planar graphs,
he demonstrated that every graph $G \in \mathcal{G}$ with $\Delta(G) \geq 17$ admits an equitable $\Delta(G)$-coloring. 
Subsequently, Zhang, Wang, and Xu \cite{zbMATH06875705} refined the findings specifically for 1-planar graphs, proving that every such graph with
$\Delta(G) \geq 15$ admits an equitable $\Delta(G)$-coloring. Most recently, Cranston and Mahmoud \cite{CRANSTON2025114286} further narrowed the threshold to $\Delta(G) \geq 13$ for 1-planar graphs and proposed adapting their proof to encompass all graphs in $\mathcal{G}$.

To this end, Liu and Zhang \cite{CCXX} introduced a class $\mathcal{G}_{m_1,m_2}$ of graphs $G$ satisfying :
\begin{itemize}
    \item $e(H)\leq m_1 v(H)$ for every subgraph $H$ of $G$,
    \item $e(H)\leq m_2 v(H)$ for every bipartite subgraph $H$ of $G$,
\end{itemize}
where $e(H)$ and $v(H)$ are the number of edges and vertices in $H$.
They showed that if $m_1\leq 1.8m_2$ and $\Delta(G)\geq \frac{2}{1-\beta_*}m_1$,
where $\beta_*$ is a real root of the function
\[ 
\frac{2(1-x)(1+x)^2}{x(2+x)}=\frac{m_1}{m_2},
\]
then $G$ is equitably $r$-colorable for every $r\geq \Delta(G)$.

In this paper, we focus on two subclasses of 1-planar graphs: \textit{IC-planar graphs}, where crossings are independent (i.e.,\,no two crossed edges share a common endpoint), and \textit{NIC-planar graphs}, where crossings are nearly independent (i.e.,\,no two pairs of crossed edges share two endpoints).
IC-planar graphs were initially studied by Kr\'al' and Stacho \cite{zbMATH05814476} under the designation "plane graphs with independent crossings." However, numerous recent studies (e.g., \cite{zbMATH07383905, zbMATH06888432,zbMATH07796524}) have adopted the terminology of IC-planar graphs, as advocated by Zhang \cite{zbMATH06378992}, who also introduced the concept of NIC-planar graphs in the same work. 

Zhang \cite{zbMATH06602931} demonstrated that every IC-planar graph $G$ with maximum degree $\Delta(G) \geq 12$ and every NIC-planar graph $G$ with maximum degree $\Delta(G) \geq 13$ admits an equitable $\Delta(G)$-coloring. It is known that all IC-planar graphs belong to the class $\mathcal{G}_{\frac{7}{2},\frac{9}{4}}$, and all NIC-planar graphs belong to the class $\mathcal{G}_{\frac{18}{5},\frac{5}{2}}$ (refer to the lemmas in the subsequent section for details). Applying the aforementioned result of Liu and Zhang \cite{CCXX}, we are unable to derive tighter bounds than those already established.
However, we can establish the following result through much more careful analyses. 

\begin{thm} \label{thm:main}
    If $G$ is an IC-planar graph with $\Delta(G)\geq 10$ or an NIC-planar graph with $\Delta(G)\geq 11$, then $G$ has an equitable coloring with $r$ colors for every $r\geq \Delta(G)$.
\end{thm}

\section{Proof of Theorem \ref{thm:main}}  \label{sec:xxlovecc-2}

Drawing inspiration from the techniques presented in \cite{CRANSTON2025114286} and \cite{zbMATH07840571}, we adapt and integrate them with numerous new ideas to align with our research objectives in this particular context. We begin with some useful lemmas.

\begin{lem} {\rm \cite{zbMATH06137073}}  \label{lem:size-IC}
    Every $n$-vertex IC-planar graph has at most $3.5n-7$ edges.
\end{lem}

\begin{lem} {\rm \cite{zbMATH07561382}}  \label{lem:size-bipIC}
    Every $n$-vertex bipartite IC-planar graph has at most $2.25n-4$ edges.
\end{lem}

\begin{lem} {\rm \cite{zbMATH06378992}}  \label{lem:size-NIC}
    Every $n$-vertex NIC-planar graph has at most $3.6n-7.2$ edges. 
\end{lem}

\begin{lem} {\rm \cite{zbMATH06378992}}  \label{lem:6-ge-NIC}
    Every NIC-planar graph is $6$-degenerate. 
\end{lem}

\begin{lem} {\rm \cite{zbMATH07561382}}  \label{lem:size-bipNIC}
    Every $n$-vertex bipartite NIC-planar graph has at most $2.5n-5$ edges.
\end{lem}

\begin{lem} {\rm \cite{CCXX}} \label{love}
Let $m_2\leq m_1\leq 1.8m_2$. If every graph $H$ in $\mathcal{G}_{m_1,m_2}$ with $r \mid v(H)$ is equitably $r$-colorable, then every graph $G\in \mathcal{G}_{m_1,m_2}$ is equitably $r$-colorable.
\end{lem}

Now we are ready to prove Theorem \ref{thm:main}. In this proof, we do not differentiate between whether $G$ is IC-planar or NIC-planar, except when such distinction is imperative. 

By Lemmas \ref{lem:size-IC}, \ref{lem:size-bipIC}, \ref{lem:size-NIC}, \ref{lem:size-bipNIC}, and \ref{love}, we shall assume that $G$ has $n:=sr$ vertices and prove the theorem by induction on the number of edges. Since $G$ is 6-degenerate by Lemmas \ref{lem:size-IC} and \ref{lem:6-ge-NIC}, we choose an edge $xy$ such that $d_G(x)\leq 6$, then remove it from $G$ and denote the resulting graph by $G'$. Now, by the induction hypothesis, $G'$ has an equitable $r$-coloring $\varphi$. If $\varphi(x)\neq \varphi(y)$, then $\varphi$ is also an equitable $r$-coloring of $G$ and we are done.
So we assume $\varphi(x)=\varphi(y)$.

Now, define $H$ as $G' - x$, and consider the restriction of the coloring $\varphi$ to the graph $H$. 
 This restriction results in an equitable $r$-coloring of $H$, which we continue to denote by $\varphi$. In this coloring, all color classes except the one containing $y$ (which has size $s - 1$ and is referred to as a \textit{small class}) have size $s$. Let $V_1$ denote the color class containing $y$, and let $V_2, \ldots, V_r$ represent the remaining color classes.

 Construct a directed graph $\mathcal{D} := \mathcal{D}(\varphi)$ on the vertex set $\{V_1, V_2, \ldots, V_r\}$ such that the directed edge $\overrightarrow{V_iV_j}$ belongs to the edge set $E(\mathcal{D})$ if and only if there exists a vertex $v \in V_i$ for which $e_H(v, V_j) = 0$. We refer to such a vertex $v$ as \textit{movable} to $V_j$ and state that $v$ \textit{witnesses} the edge $\overrightarrow{V_iV_j}$.
 
If $P := \overrightarrow{V_{j_1}V_{j_2} \cdots V_{j_k}}$ is a directed path in $\mathcal{D}$ and $v_i$ is a vertex in $V_{j_i}$ such that $v_i$ witnesses $\overrightarrow{V_{j_i}V_{j_{i+1}}}$, then \textit{switching witnesses along $P$} involves moving $v_i$ to $V_{j_{i+1}}$ for every $1 \leq i < k$. This operation reduces the size of $V_{j_1}$ by one and increases the size of $V_{j_k}$ by one, while keeping the sizes of the intermediate vertices (color classes) unchanged.
 
For a particular $V_j$, if $j = 1$ or there exists a directed path from $V_j$ to $V_1$, we say that $V_j$ is \textit{accessible}. If $V_j$ is an accessible class and for each accessible class $V_k$ with $k \neq j$ there is a directed path from $V_k$ to $V_1$ in $\mathcal{D} - V_j$, then we refer to such $V_j$ as a \textit{terminal}. For convenience, we also consider $V_1$ as a terminal.

Let $\mathcal{A}$ denote the collection of all accessible classes, and let $\mathcal{B}$ represent the set of remaining classes.
Define
\begin{align*}
    a := a(\varphi) &= |\mathcal{A}|, \\
    b := b(\varphi) &= |\mathcal{B}|.
\end{align*}
It is noted that $a$ is uniquely determined by $\varphi$ and satisfies $a \geq 1$. By selecting an appropriate $\varphi$, we can ensure that $a$ attains its maximum value.
 
Let $A = \{v \in V(H) \mid \exists V_j \in \mathcal{A} \text{ such that } v \in V_j\}$ and $B = \{v \in V(H) \mid \exists V_j \in \mathcal{B} \text{ such that } v \in V_j\}$. Consequently,
\begin{align*}
    |A| &= as - 1, \\
    |B| &= bs.
\end{align*}

\begin{claim} \label{claim:a<=3}
    $a\leq 3$.
\end{claim}

\begin{proof}
Suppose by contradiction that $a\geq 4$. If $a\geq 7$, then $x$ is not adjacent to some class, namely $V_i$, in $\mathcal{A}$ because $d_G(x)\leq 6$. Now, moving $x$ to $V_i$ and switching witnesses along the path from $V_i$ to $V_1$, we obtain an equitable $r$-coloring of $G$.

Therefore, we have $4\leq a\leq 6$. Since each class $V_j\in \mathcal{B}$ is not accessible, for every $u\in V_j$ we have $e_H(u,V_i)\geq 1$ for every $V_i\in \mathcal{A}$. This implies
\[e_H(A,B)\geq |\mathcal{A}|\cdot |B| = abs = a(r-a)s.\]
On the other hand, since the graphs induced by the edges between $A$ and $B$ is bipartite, by Lemmas \ref{lem:size-bipIC} and \ref{lem:size-bipNIC}, we have
\[
e_H(A,B) \leq 
\begin{cases} 
2.25n - 4 = 2.25(rs - 1) - 4, & \text{if } G \text{ is IC-planar}, \\
2.5n - 5 = 2.5(rs - 1) - 5, & \text{if } G \text{ is NIC-planar}.
\end{cases}
\]
Combining the above two inequalities, we have  
\[
\begin{cases} 
2.25r-a(r-a)  \geq 6.25/s, & \text{if } G \text{ is IC-planar}, \\
2.5r-a(r-a)  \geq 7.5/s, & \text{if } G \text{ is NIC-planar}.
\end{cases}
\]
However, given that $r \geq 10$ when $G$ is IC-planar and $r \geq 11$ when $G$ is NIC-planar, this inequality fails to hold in the case where $4 \leq a \leq 6$.
\end{proof}

A subset $\mathcal{C} \subseteq V(\mathcal{D})$ is called a \textit{strong component} if, for any two distinct classes $V_i, V_j \in \mathcal{C}$, there exists a $(V_i, V_j)$-path in $\mathcal{D}$. We demonstrate that the directed subgraph induced by $\mathcal{B}$ contains a large strong component.
 
\begin{claim} \label{claim:strong-component}
There exists a subset $\mathcal{C} \subseteq \mathcal{B}$ such that $|\mathcal{C}| \geq r - 3$ and $\mathcal{C}$ forms a strong component in the subgraph $\mathcal{D}[\mathcal{B}]$.
\end{claim}

\begin{proof}
Suppose, to the contrary, that every strong component in $\mathcal{D}[\mathcal{B}]$ has at most $r-4$ classes. 

If there exists a strong component $\mathcal{C}_0$ in $\mathcal{D}[\mathcal{B}]$ such that $|\mathcal{C}_0| \geq 4$, then let $\mathcal{Z} := \mathcal{C}_0$. Otherwise, if every strong component in $\mathcal{D}[\mathcal{B}]$ has at most three classes, then assume that $\mathcal{D}[\mathcal{B}]$ has exactly $N$ strong components $\mathcal{C}_1, \ldots, \mathcal{C}_N$, where $|\mathcal{C}_1| \leq \cdots \leq |\mathcal{C}_N|$. Let $M$ be the smallest integer such that 
\[
|\mathcal{C}_1| + \cdots + |\mathcal{C}_M| \geq 4,
\]
and define $\mathcal{Z} := \bigcup_{i=1}^M \mathcal{C}_i$. 
 If $|\mathcal{Z}| \geq r - 3$, then 
\[
|\mathcal{C}_1| + \cdots + |\mathcal{C}_{M-1}| \geq r - 3 - |\mathcal{C}_M| \geq r - 6 \geq 4,
\]
which contradicts the minimality of $M$. 
Therefore, in either case, we can find a union $\mathcal{Z}$ of some strong components in $\mathcal{D}[\mathcal{B}]$ such that 
\[
4 \leq z := |\mathcal{Z}| \leq r - 4.
\]

Define the set $Z := \{v \in V(H) \mid \exists V_j \in \mathcal{Z} {\rm ~such ~that~} v\in V_j\}$.
For every class $V_i \in \mathcal{B} \setminus \mathcal{Z}$ and every class $V_j \in \mathcal{Z}$, either $\overrightarrow{V_iV_j} \notin E(\mathcal{D})$ or $\overrightarrow{V_jV_i} \notin E(\mathcal{D})$. This implies that for every $v \in V_i$, either $e_H(v, V_j) \geq 1$ or, equivalently, for every $v \in V_j$, $e_H(v, V_i) \geq 1$. In either case, we have
$e_H(V_i, V_j) \geq s$. Consequently,
\begin{equation*}
e_H(Z, \mathcal{B} \setminus \mathcal{Z}) \geq z(b - z)s.
\end{equation*}
Since $\mathcal{Z} \subseteq \mathcal{B}$, for every vertex $v$ in every class $V_i \in \mathcal{Z}$, we have $e_H(v, V_j) \geq 1$ for every $V_j \in \mathcal{A}$. It follows that
\begin{equation*}
e_H(Z, \mathcal{A}) \geq azs.
\end{equation*}

Since the graph induced by the edges between $Z$ and $A\cup B\setminus Z$ is bipartite, combining the above two inequalities, we conclude
\begin{align*} 
      (r-z)zs=(a+b-z)zs=z(b-z)s+azs\leq  e_H(Z,A\cup B\setminus Z)\leq 
      \begin{cases} 
2.25rs, & \text{if } G \text{ is IC-planar}, \\
2.5rs, & \text{if } G \text{ is NIC-planar}.
\end{cases}
\end{align*}
However, given that $r \geq 10$ when $G$ is IC-planar and $r \geq 11$ when $G$ is NIC-planar, this inequality fails to hold in the case where $4 \leq z \leq r-4$.
\end{proof}

Consider an edge $uv$ such that $u \in V_i \in \mathcal{A}$ and $v \in V_j \in \mathcal{B}$. 
If $e_H(v, V_i) = 1$ (i.e., $u$ is the unique neighbor of $v$ in $V_i$), then we define the following:
\begin{itemize}
    \item $uv$ is a \textit{solo edge}.
    \item $u$ is a \textit{solo vertex}.
    \item $v$ is a \textit{solo neighbor} of $u$.
\end{itemize}
Furthermore, if there are two solo edges $uv$ and $uw$, where $u \in \mathcal{A}$ and $v, w \in \mathcal{B}$, such that $vw \notin E(H)$, then we define:
\begin{itemize}
    \item $v$ and $w$ are \textit{nice solo neighbors} of $u$.
\end{itemize}
If $uv$ is a solo edge with $u \in V_i \in \mathcal{A}$, then $V_i + v - u$ remains an independent set. This property will be frequently utilized in subsequent arguments.
 
Consider a vertex $u \in \mathcal{A}$. We define the following:
\begin{itemize}
    \item Let $Q(u)$ denote the set of \textit{solo neighbors} of $u$.
    \item Let $Q'(u)$ denote the set of \textit{nice solo neighbors} of $u$.
    \item Define $q(u) = |Q(u)|$ as the \textit{solo degree} of $u$.
    \item Define $q'(u) = |Q'(u)|$ as the \textit{nice solo degree} of $u$.
\end{itemize}

\begin{claim}\label{claim:solo-at-least-2-neighbors}
Let $v\in V_i\in \mathcal{A}$. If $i=1$ or there is a vertex $v'\in V_i\setminus \{v\}$ such that $v'$ is movable to $V_1$, then for each $V_j\in \mathcal{B}$ containing a nice solo neighbor of $v'$, 
$v'$ has at least two neighbors in $V_j$.
\end{claim}

\begin{proof}
Assume $w \in Q'(v') \cap V_j$. We claim that $v'$ has another neighbor in $V_j$ besides $w$. Suppose by contradiction that $e_H(v',V_j)=1$.
Since $w$ is a nice solo neighbor of $v'$, there is another nice solo neighbor $z$ of $v'$ in some other class $V_{j'}$ with $j'\neq j$ in $\mathcal{B}$ such that $wz \notin E(H)$. In this case, we move $v'$ to $V_j$ and $w$ to $V_i$. This constructs a new equitable $r$-coloring with the unique small class still being $V_1$. Now, $V_j\cup \{v'\} \setminus \{w\}$ is accessible to $V_i\cup \{w\} \setminus \{v'\}$ because of $v'$, $V_{j'}$ is also accessible to $V_i\cup \{w\} \setminus \{v'\}$ by the fact that $v'z$ is an solo edge and $zw\not\in E(H)$, and $V_i\cup \{w\} \setminus \{v'\}$ is accessible to $V_1$ because of $v$. This implies that the new $a$ value of the coloring is at least 4, contradicting the maximality of $a$ by Claim \ref{claim:a<=3}.
\end{proof}

\begin{claim}\label{claim:position-of-nicesolo-neighbor}
Let $v\in V_i\in \mathcal{A}$. If $i=1$ or there is a vertex $v'\in V_i\setminus \{v\}$ such that $v'$ is movable to $V_1$, then $Q'(v)\subseteq \mathcal{B} \setminus \mathcal{C}$ provided that there is a class in $\mathcal{C}$ that does not containing any neighbor of $v$.
\end{claim}

\begin{proof}
Let $V_\ell$ be the class in $\mathcal{C}$ that does not containing any neighbor of $v$.
Suppose by contradiction that $v'$ has a nice solo neighbor $w$ in some class $V_j$ in $\mathcal{C}$. 
Let $z\in V_{j'}$ (where $j'=j$ is feasible) be another nice solo neighbor of $v'$ such that $zw\not\in E(H)$.  
Since $\mathcal{C}$ induces a strong component, there is a directed path $P$ from $V_\ell$ to $V_j$ in $\mathcal{D}[\mathcal{B}]$. 
We assume that $P$ does not pass through $V_{j'}$ when $V_{j'} \in \mathcal{C}$, as exchanging $V_j$ and $V_{j'}$ would be possible otherwise.

Now, we move $v$ to $V_\ell$, $w$ to $V_i$, and switch witness along $P$. This gives an equitable $r$-coloring with the unique small class still being $V_1$. 
Since $vz$ is a solo edge and $zw\not\in E(G)$, $V_{j'}$ is accessible to $V_i\cup \{w\} \setminus \{v'\}$. Because of $v$, $V_i\cup \{w\} \setminus \{v'\}$ is still accessible to $V_1$.
Therefore, the new $a$ value is at least one greater than its preceding value, a contradiction to the maximality of $a$.
\end{proof}

Next, We establish the relationship between the solo degree and nice solo degree as follows.

\begin{claim} \label{claim:q-and-q'}
If $G$ is IC-planar and $q(u)\geq 6$, or $G$ is NIC-planar and $q(u)\geq 7$, then $q'(u)\geq q(u)-2$.
\end{claim}

\begin{proof}
Suppose $q'(u)\leq q(u)-3$.
For every pair of $v \in Q(u) \setminus Q'(u)$ and $w \in Q(u)$, $vw \in E(H)$ by the definition of nice solo neighbors. 
Counting edges in the subgraph of $H$ induced by $u\cup Q(u)$, we have
\begin{align*} 
    e(H[u\cup Q(u)]) &=\binom{q(u)+1}{2}-e(\overline{H[Q'(u)]})\\
                     &\geq \binom{q(u)+1}{2}-\binom{q'(u)}{2}\geq \binom{q(u)+1}{2}-\binom{q(u)-3}{2}=4q(u)-6.
\end{align*}
On the other hand, by Lemmas \ref{lem:size-IC} and \ref{lem:size-NIC},
\[
e(H[u\cup Q(u)]) \leq 
\begin{cases} 
3.5(q(u)+1) - 7, & \text{if } G \text{ is IC-planar}, \\
3.6(q(u)+1) - 7.2, & \text{if } G \text{ is NIC-planar}.
\end{cases}
\]
Combining the above two inequalities together, we conclude that
\[
q(u) \leq 
\begin{cases} 
5, & \text{if } G \text{ is IC-planar}, \\
6, & \text{if } G \text{ is NIC-planar}.
\end{cases}
\]
This completes the proof.
\end{proof}

Fix a class $V_i\in\mathcal{A}$ and a subset $\mathcal{W}\subseteq\mathcal{B}$ (with the possibility that $\mathcal{W}=\emptyset$). Define $W = \{v\in V(H)\mid\exists V_j\in\mathcal{W}  {\rm ~such ~that~} v\in V_j\}$.
For each edge $uv$ where $u\in V_i\in\mathcal{A}$ and $v\in B$, define its weight with respect to $\mathcal{W}$:
\begin{align*}
f_{\mathcal{W}}(uv):=
\begin{cases}
\frac{0.5}{e_H(v,V_i)}&\text{if }v\in W,\\
\frac{1}{e_H(v,V_i)}&\text{if }v\in B\setminus W.
\end{cases}
\end{align*}
For every $u\in A$, let
\[
f_{\mathcal{W}}(u)=\sum_{v\in N_B(u)}f_{\mathcal{W}}(uv).
\]
It immediately follows
\begin{align} \label{eq:sum}
   \notag \sum_{u\in V_i}f_{\mathcal{W}}(u)=\sum_{u\in V_i}\sum_{v\in N_B(u)}f_{\mathcal{W}}(uv)&=\sum_{\substack{ u\in V_i \\ v \in B \setminus W \\ uv \in E(H)  }}f_{\mathcal{W}}(uv)+\sum_{\substack{ u\in V_i \\ v \in  W \\ uv \in E(H) }} f_{\mathcal{W}}(uv)\\
 &= |B\setminus W|+0.5|W|=|B|-0.5|W|.
\end{align}
The rest of the proof is divided into three distinct parts according to the value of $a$.

\vspace{3mm}
\noindent \textbf{Case 1:} $a=3$
\vspace{3mm}

Let $\mathcal{A}=\{V_1,V_2,V_3\}$ and $\mathcal{B}=\{V_4,V_5,\ldots,V_r\}$.
We may assume that $\overrightarrow{V_3V_1},\overrightarrow{V_2V_1}\in E(\mathcal{D})$. If it is not the case, then assume, without loss of generality, that $\overrightarrow{V_3V_2},\overrightarrow{V_2V_1}\in E(\mathcal{D})$. Since $\overrightarrow{V_2V_1}\in E(\mathcal{D})$, there is a vertex $u\in V_2$ that are movable to $V_1$. Now, let $V_1:=V_2 \setminus \{u\}$ and $V_2:=V_1\cup \{u\}$. It follows that  $\overrightarrow{V_3V_1},\overrightarrow{V_2V_1}\in E(\mathcal{D})$. 

Let $v\in V_2$ be a vertex movable to $V_1$. If $v$ has a solo neighbor $u$ in some class in $\mathcal{B}$, namely $V_4$, then we move $u$ to $V_2$, and $v$ to $V_1$. We come to a new equitable $r$-coloring whose unique small class is $V_4\setminus \{u\}$. Since $\mathcal{D}[\mathcal{B}]$ is a strong component by Claim \ref{claim:strong-component}, $V_5,\ldots,V_r$ are all accessible to $V_4\setminus \{u\}$. Hence, the new $a$ value of this equitable $r$-coloring is at least $r-3$, bigger than 3, contradicting the maximality of $a$. This implies $q(v)=0$ and that every vertex $u\in V_2$ with $q(u)\geq 1$ has at least one neighbor in $V_1$. Furthermore, $f_{\emptyset}(v) \leq \frac{1}{2} d_{H}(v)\leq \frac{1}{2}r<r-3$ because every neighbor of $v$ in $B$ has at least two neighbors in $V_2$.

Since 
\[\frac{\sum_{u\in V_2} f_{\emptyset}(u)}{|V_2|}=\frac{s(r-3)}{s}=r-3\]
by \eqref{eq:sum},
there exists $v'\in V_2$ such that \[f_{\emptyset}(v')>r-3.\] This implies that $v'$ is not movable to $V_1$, and thus $v'$ has at least one neighbor in $V_1$. On the other hand,
\[f_{\emptyset}(v')\leq \frac{1}{2}\bigg(d_H(v')-1-q(v')\bigg)+q(v') \leq \frac{1}{2} (q(v')+r-1).\]
It follows $q(v')>r-5$. If $G$ is IC-planar, then $q(v')\geq 6$; and if $G$ is NIC-planar, then $q(v')\geq 7$. In each case, we have $q'(v')\geq 4$ by Claim \ref{claim:q-and-q'}.


If $v'$ is movable to $V_3$, then move $v'$ to $V_3$ and move the witness from $V_3$ to $V_1$. This results in a new equitable $r$-coloring with $V_2\setminus \{v'\}$ being the unique small class. Clearly,
every class in $\mathcal{B}$ containing solo neighbors of $v'$ is accessible to $V_2\setminus \{v'\}$, and there are $q(v')\geq 6$ such classes. Now, we have a new $a$ value of at least 7, contradicting the maximality of $a$. Therefore, $v'$ has at least one neighbor in $V_3$.

Let $N_1$ be the number of classes in $\mathcal{B}$ containing exactly one neighbor of $v'$ and let $N_2$ be the number of classes in $\mathcal{B}$ containing at least two neighbors of $v'$.
If $|N_1|+|N_2|=|\mathcal{B}|$, then (recall that we had shown that $v'$ has at least one neighbor in both $V_1$ and $V_3$)
\[r\geq d_H(v')\geq 2+N_1+2N_2=r-1+N_2,\]
implying $N_2\leq 1$. However, $N_2\geq \lceil \frac{q'(v')}{2}\rceil\geq 2$ by Claim \ref{claim:solo-at-least-2-neighbors}, a contradiction. Therefore, $|N_1|+|N_2|<|\mathcal{B}|$, meaning that there exists a class $V_j\in \mathcal{B}$ such that $V_j$ 
has no neighbor of $v'$. 

Now we are able to close this case. Since $\mathcal{D}[\mathcal{B}]$ is a strong component by Claim \ref{claim:strong-component}, $Q'(v')=\emptyset$ by Claim \ref{claim:position-of-nicesolo-neighbor}. This contradicts the fact that $q'(v')\geq 4$.


\vspace{3mm}
\noindent \textbf{Case 2:} $a=2$
\vspace{3mm}

Let $\mathcal{A}=\{V_1,V_2\}$ and $\mathcal{B}=\{V_3,V_4,\ldots,V_r\}$. Now, $\overrightarrow{V_2V_1}\in E(\mathcal{D})$.

Let $v$ be a vertex in $V_2$ that are movable to $V_1$. We move $v$ to $V_1$ and obtain an equitable $r$-coloring with $V_2\setminus \{v\}$ being the unique small class.
If $v$ has a solo neighbor $u$, then the class containing $u$ is accessible to $V_2 \setminus \{v\}$.
Trivially, $V_1\cup \{v\}$ is also accessible to $V_2 \setminus \{v\}$.  Therefore, the new $a$ value is at least 3, a contradiction to the maximality of $a$. This implies $q(v)=0$ and that every vertex $u\in V_2$ with $q(u)\geq 1$ has at least one neighbor in $V_1$. 
Furthermore, $f_{\emptyset}(v) \leq \frac{1}{2} d_{H}(v)\leq \frac{1}{2}r<r-2$ because every neighbor of $v$ in $B$ has at least two neighbors in $V_2$.

Since 
\[\frac{\sum_{u\in V_2} f_{\emptyset}(u)}{|V_2|}=\frac{s(r-2)}{s}=r-2\]
by \eqref{eq:sum},
there exists $v'\in V_2$ such that \[f_{\emptyset}(v')>r-2.\]
This suggests that $v'$ cannot be moved to $V_1$, indicating $v'$ possesses at least one neighbor in $V_1$. On the other hand,
\[f_{\emptyset}(v')\leq \frac{1}{2}\bigg(d_H(v')-1-q(v')\bigg)+q(v') \leq \frac{1}{2} (q(v')+r-1).\]
It follows $q(v')>r-3 \geq 7$ and thus $q'(v')\geq 6$ by Claim \ref{claim:q-and-q'}.

By Claim \ref{claim:strong-component}, $\mathcal{D}[\mathcal{B}]$ has a strong component $\mathcal{D}[\mathcal{C}]$ containing at least $r-3$ classes. Without loss of generality, assume $\{V_4,V_5,\ldots,V_r\}\subseteq \mathcal{C}$.
Let $N_1$ be the number of classes in $\mathcal{B}$ containing exactly one neighbor of $v'$ and let $N_2$ be the number of classes in $\mathcal{B}$ containing at least two neighbors of $v'$.
If $|N_1|+|N_2| \geq |\mathcal{B}|-1$, then by Claim \ref{claim:solo-at-least-2-neighbors} (recall that we had shown that $v'$ has at least one neighbor in both $V_1$)
\[r\geq d_H(v')\geq 1+N_1+2N_2\geq |\mathcal{B}|+N_2 \geq |\mathcal{B}|+\bigg\lceil \frac{q'(v')}{2}\bigg\rceil \geq (r-2)+3=r+1,\]
a contradiction. Therefore, $|N_1|+|N_2| \leq |\mathcal{B}|-2=r-4$. This suggests that within $\{V_4,V_5,\ldots,V_r\}$, there exist at least one class, say $V_r$, that does not contain any neighbors of $v'$.
It follows $Q'(v')\subseteq V_3$ by Claim \ref{claim:position-of-nicesolo-neighbor}.


Now,
\[\frac{\sum_{u\in V_2} f_{V_3}(u)}{|V_2|}=\frac{(r-2-0.5)s}{s}=r-2.5\]
by \eqref{eq:sum}, 
\begin{align*}
    f_{V_3}(v')&\leq 0.5q'(v')+\left(q(v')-q'(v')\right)+0.5(r-q(v'))\\
               &=0.5(q(v')-q'(v')+r)\leq 0.5(2+r)<r-2.5
\end{align*}
by Claim \ref{claim:q-and-q'}, and
\[ f_{V_3}(v) \leq f_{\emptyset}(v) \leq \frac{1}{2} d_{H}(v)\leq \frac{1}{2}r<r-2.5\] because $q(v)=0$.
Therefore, there is a vertex $v''\in V_2\setminus \{v,v'\}$ such that \[f_{V_3}(v'')>r-2.5.\]
Moreover, $q(v_3)\geq 1$ because otherwise $f_{V_3}(v'')\leq f_{\emptyset}(v'')\leq 0.5d_{H}(v'')\leq 0.5r<r-2.5$.
This further implies that $v''$ has at least one neighbor in $V_1$.
On the other side, let $m=|Q(v'')\setminus V_3|$. We have
\[f_{V_3}(v'')\leq \frac{1}{2} \bigg(d_H(v'')-1-m\bigg)+m\leq \frac{1}{2}(r+m-1).\]
Therefore, we deduce $q(v'')\geq m>r-4\geq 6$ and thus $q'(v'')\geq 5$ by Claim \ref{claim:q-and-q'}.


Let $N_1$ denote the number of classes in $\mathcal{B}$ with exactly one neighbor of $v''$, and $N_2$ the number with at least two neighbors of $v''$.
If $|N_1|+|N_2| \geq |\mathcal{B}|-1$, then by Claim \ref{claim:solo-at-least-2-neighbors} (given $v''$ has a neighbor in both $V_1$),
\[
r\geq d_H(v'')\geq 1+N_1+2N_2\geq |\mathcal{B}|+N_2 \geq |\mathcal{B}|+\bigg\lceil \frac{q'(v'')}{2}\bigg\rceil \geq (r-2)+3=r+1,
\]
a contradiction. Thus, $|N_1|+|N_2| \leq |\mathcal{B}|-2=r-4$. This implies $\{V_4,V_5,\ldots,V_r\}$ contains at least one class with no neighbors of $v''$. 

By Claim \ref{claim:position-of-nicesolo-neighbor}, we have $Q'(v'') \subseteq V_3$. However, it is not the case because otherwise
\[f_{V_3}(v'')\leq 0.5q'(v'')+\bigg(r-1-q'(v'')\bigg)=r-1-0.5q'(v'')\leq r-3.5<r-2.5,\]
a contradiction.

\vspace{3mm}
\noindent \textbf{Case 3:} $a=1$
\vspace{3mm}

By Claim \ref{claim:q-and-q'}, we have
\[\frac{\sum_{u\in V_1} f_{\emptyset}(u)}{|V_1|}=\frac{(r-1)s}{s-1}>r-1,\]
so there is a vertex $v\in V_1$ such that
\[f_{\emptyset}(v)>r-1.\]
On the other hand,
\[f_{\emptyset}(v)\leq \frac{1}{2}\bigg(d_H(v)-q(v)\bigg)+q(v) \leq \frac{1}{2} (q(v)+r).\]
It follows $q(v) \geq r-1\geq 9$ and thus $q'(v)\geq 7$  by Claim \ref{claim:q-and-q'}.

We assume, by Claim \ref{claim:strong-component}, that $\mathcal{C}=\{V_4,V_5,\ldots,V_r\}$ is contained in a strong component of $\mathcal{D}[\mathcal{B}]$.
Let $N_1$ denote the number of classes in $\mathcal{B}$ with exactly one neighbor of $v$, and $N_2$ the number with at least two neighbors of $v$.
If $|N_1|+|N_2| \geq |\mathcal{B}|-2$, then by Claim \ref{claim:solo-at-least-2-neighbors},
\[
r\geq d_H(v)\geq N_1+2N_2\geq |\mathcal{B}|-2+N_2 \geq |\mathcal{B}|-2+\bigg\lceil \frac{q'(v)}{2}\bigg\rceil \geq (r-3)+4=r+1,
\]
a contradiction. Thus, $|N_1|+|N_2| \leq |\mathcal{B}|-3=r-4$. This indicates $\mathcal{B}$ has at least three classes with no neighbors of $v$. 
Now, we have 
$Q'(v) \subseteq V_2\cup V_3$ by Claim \ref{claim:position-of-nicesolo-neighbor}.

Let $\mathcal{W}=\{V_2,V_3\}$. Since
\[\frac{\sum_{u\in V_1} f_{\mathcal{W}}(u)}{|V_1|}=\frac{(r-1)s-0.5(2s)}{s-1}>r-2\]
by \eqref{eq:sum}, and
\begin{align*}
    f_{\mathcal{W}}(v)&\leq 0.5q'(v)+\left(q(v)-q'(v)\right)+0.5(r-q(v))\\
               &=0.5(q(v)-q'(v)+r)\leq 0.5(2+r)<r-2,
\end{align*}
there exists $v'\in V_1\setminus \{v\}$ such that $f_{\mathcal{W}}(v')>r-2$.

Let $m=|Q(v')\setminus \mathcal{W}|$. Then, we have
\[f_{\mathcal{W}}(v')\leq \frac{1}{2} \bigg(d_H(v'')-m\bigg)+m\leq \frac{1}{2}(r+m).\]
Therefore, we deduce $q(v')\geq m>r-4\geq 6$ and thus $q'(v')\geq 5$ by Claim \ref{claim:q-and-q'}.

If at least five  vertices of $Q'(v')$ are in $\mathcal{W}$, then \[f_{\mathcal{W}}(v')\leq d_H(v')-0.5(5)\leq r-2.5<r-2,\] a contradiction.
So, at least one vertex of $Q'(v')$ is in some class of $\mathcal{B}\setminus \mathcal{W}$.

Let $N_0$ denote the number of classes in $\mathcal{B}$ with no neighbor of $v'$, $N_1$ the number of classes in $\mathcal{B}$ with exactly one neighbor of $v'$, and $N_2$ the number with at least two neighbors of $v'$. We have
\[
r\geq d_H(v')\geq N_1+2N_2\geq |\mathcal{B}|-N_0+N_2 \geq |\mathcal{B}|-N_0+\bigg\lceil \frac{q'(v')}{2}\bigg\rceil
\]
by Claim \ref{claim:solo-at-least-2-neighbors}, implying
\begin{align*}
    N_0 \geq \bigg\lceil \frac{q'(v')}{2}\bigg\rceil-1.
\end{align*}

If some class of $\mathcal{B} \setminus \mathcal{W}$ contains no neighbor of $v'$, then $Q'(v')\subseteq \mathcal{W}$ by Claim \ref{claim:position-of-nicesolo-neighbor}, contradicting the fact that $Q'(v')\cap \mathcal{B}\setminus \mathcal{W} \neq \emptyset$ we had just proved.  Therefore, $N_0\leq 2$, implying $q'(v')\leq 6$. On the other hand, since $q'(v')\geq 5$, we have $N_0\geq 2$. It follows $N_0=2$ and 
\[N_H(v')\cap (V_2\cup V_3)=\emptyset.\]

Since $m\geq r-3\geq 7$, there exists $z\in \left(Q(v') \setminus \mathcal{W}\right) \setminus Q'(v')$. Assume by symmetry that $z\in V_r$.
Since $z\not\in Q'(v')$, $z$ is adjacent to all other vertices in $Q(v') \setminus \mathcal{W}$. It follows that $z$ has at least $r-4$ neighbors in $\mathcal{B} \setminus \mathcal{W}$. Since $z$ has a neighbor $v'$ in $V_1$, $z$ has at most three neighbors in $\mathcal{W}$.

Given that $Q'(v) \subseteq \mathcal{W}$ and $q'(v) \geq 7$, $Q'(v)$ contains at least four vertices non-adjacent to $z$. By Pigeonhole, we choose two of them, namely $z_1$ and $z_2$, that are all in $V_2$.
Recall that $\mathcal{B}$ has at least three classes with no neighbors of $v$. We may assume $V_i$ has no neighbors of $v$ for some $4\leq i\leq r$. Since $V_i,V_r\in \mathcal{C}$, there is a directed path from $V_i$ to $V_r$ in $\mathcal{D}[\mathcal{B}]$ (if $i=r$, then we simply regard $P$ as an empty structure).

Now, we complete the proof in the following way. We move $v$ to $V_i$, $v'$ to $V_2$, $z$ and $z_1$ to $V_1$, and then switch witness along $P$. 
This creates a new equitable $r$-coloring of $H$ with the unique small class being $V_1 \cup \{z,z_1\} \setminus \{v,v'\}$. However, $V_2\cup \{v'\} \setminus \{z_1\}$ is accessible to $V_1 \cup \{z,z_1\} \setminus \{v,v'\}$ because of $z_2$ (note that $vz_2$ is a solo edge and $z_1z_2,zz_2\not\in E(H)$). This implies that the new $a$ value of at least 2, contradicting the maximality of $a$.

\section{Concluding Remarks}

The proof of Theorem \ref{thm:main} does not demonstrate significant dependence on the structural attributes of IC- or NIC-planar graphs, with the exception of the initial five lemmas presented in Section \ref{sec:xxlovecc-2}. It is noteworthy that IC-planar graphs are classified within the set $\mathcal{G}_{\frac{7}{2},\frac{9}{4}}$, while NIC-planar graphs fall under the category $\mathcal{G}_{\frac{18}{5},\frac{5}{2}}$. Employing a nearly identical proof methodology, we can establish the subsequent results:

\begin{thm} \label{I-love-u-809}
    If $G\in \mathcal{G}_{\frac{7}{2},\frac{9}{4}}$ and $\Delta(G)\geq 10$, then $G$ has an equitable $r$-coloring for every $r\geq \Delta(G)$.
\end{thm}

\begin{thm} \label{I-love-u-829}
    If $G\in \mathcal{G}_{\frac{18}{5},\frac{5}{2}}$ and $\Delta(G)\geq 11$, then $G$ has an equitable $r$-coloring for every $r\geq \Delta(G)$.
\end{thm}

We emphasize that, although graphs in the class $\mathcal{G}_{\frac{18}{5},\frac{5}{2}}$ exhibit $7$-degeneracy, NIC-graphs specifically possess $6$-degeneracy by Lemma \ref{lem:6-ge-NIC}. Nevertheless, the $7$-degeneracy property of $\mathcal{G}_{\frac{18}{5},\frac{5}{2}}$ suffices for the validity of Claim \ref{claim:a<=3}. 
Moreover, through more meticulous computations in the proofs of Claims \ref{claim:a<=3} and \ref{claim:q-and-q'}, we are able to slightly extend the scope of Theorem \ref{I-love-u-809} to encompass graphs $G$ belonging to the class $\mathcal{G}_{\frac{7}{2},\frac{12}{5}}$, and similarly, Theorem \ref{I-love-u-829} can be extended to cover graphs $G$ within the class $\mathcal{G}_{\frac{18}{5},\frac{28}{11}}$.

\bibliographystyle{alpha}
\bibliography{ref,ref-2}

\newcommand{\etalchar}[1]{$^{#1}$}
\begin{thebibliography}{YWWL21}

\bibitem[ABK{\etalchar{+}}18]{zbMATH07561382}
Patrizio Angelini, Michael~A. Bekos, Michael Kaufmann, Maximilian Pfister, and Torsten Ueckerdt.
\newblock Beyond-planarity: {Tur{\'a}n}-type results for non-planar bipartite graphs.
\newblock In {\em 29th international symposium on algorithms and computation, ISAAC 2018, December 16--19, 2018, Jiaoxi, Yilan, Taiwan}, page~13. Wadern: Schloss Dagstuhl -- Leibniz Zentrum f{\"u}r Informatik, 2018.
\newblock Id/No 28.

\bibitem[Bra18]{zbMATH06888432}
Franz Brandenburg.
\newblock Recognizing {IC}-planar and {NIC}-planar graphs.
\newblock {\em J. Graph Algorithms Appl.}, 22(2):239--271, 2018.

\bibitem[CHL14]{chen2014equitable}
Bor-Liang Chen, Kuo-Ching Huang, and Ko-Wei Lih.
\newblock Equitable coloring of graphs with intermediate maximum degree, arXiv:1408.6046, 2014.

\bibitem[CLW94]{CHEN1994443}
Bor-Liang Chen, Ko-Wei Lih, and Pou-Lin Wu.
\newblock Equitable coloring and the maximum degree.
\newblock {\em European J. Combin.}, 15(5):443--447, 1994.

\bibitem[CM25]{CRANSTON2025114286}
Daniel~W. Cranston and Reem Mahmoud.
\newblock Equitable coloring in 1-planar graphs.
\newblock {\em Discrete Mathematics}, 348(2):114286, 2025.

\bibitem[Erd64]{erdos1959}
P.~Erd\H{o}s.
\newblock Problem 9.
\newblock In M.~Fieldler, editor, {\em Theory of Graphs and Its Applications}, page 159, Prague, 1964. Czech Academy of Science Publishing.

\bibitem[HS70]{HS1970}
A.~Hajnal and E.~Szemer\'edi.
\newblock Proof of a conjecture of {Erd\H{o}s}.
\newblock In P.~Erd\H{o}s, A.~R\'enyi, and V.T. S\'os, editors, {\em Combinatorial Theory and Its Applications}, volume~2, pages 601--623, Amsterdam, 1970. Netherlands: North-Holland.

\bibitem[KK08]{zbMATH05302004}
H.~A. Kierstead and A.~V. Kostochka.
\newblock A short proof of the {Hajnal}-{Szemer{\'e}di} theorem on equitable colouring.
\newblock {\em Comb. Probab. Comput.}, 17(2):265--270, 2008.

\bibitem[KK12]{zbMATH06081427}
H.~A. Kierstead and A.~V. Kostochka.
\newblock Every 4-colorable graph with maximum degree 4 has an equitable 4-coloring.
\newblock {\em J. Graph Theory}, 71(1):31--48, 2012.

\bibitem[KK15]{zbMATH06699287}
H.~A. Kierstead and A.~V. Kostochka.
\newblock A refinement of a result of {Corr{\'a}di} and {Hajnal}.
\newblock {\em Combinatorica}, 35(4):497--512, 2015.

\bibitem[KKMS10]{zbMATH05881225}
Henry~A. Kierstead, Alexandr~V. Kostochka, Marcelo Mydlarz, and Endre Szemer{\'e}di.
\newblock A fast algorithm for equitable coloring.
\newblock {\em Combinatorica}, 30(2):217--224, 2010.

\bibitem[KLM17]{zbMATH06782829}
Stephen~G. Kobourov, Giuseppe Liotta, and Fabrizio Montecchiani.
\newblock An annotated bibliography on 1-planarity.
\newblock {\em Comput. Sci. Rev.}, 25:49--67, 2017.

\bibitem[KLX24]{zbMATH07840571}
Alexandr Kostochka, Duo Lin, and Zimu Xiang.
\newblock Equitable coloring of planar graphs with maximum degree at least eight.
\newblock {\em Discrete Math.}, 347(6):13, 2024.
\newblock Id/No 113964.

\bibitem[KS10]{zbMATH05814476}
Daniel Kr{\'a}l' and Ladislav Stacho.
\newblock Coloring plane graphs with independent crossings.
\newblock {\em J. Graph Theory}, 64(3):184--205, 2010.

\bibitem[LY24]{zbMATH07796524}
Weichan Liu and Guiying Yan.
\newblock Weak-dynamic coloring of graphs beyond-planarity.
\newblock {\em Graphs Comb.}, 40(1):11, 2024.
\newblock Id/No 7.

\bibitem[LZ24]{CCXX}
Weichan Liu and Xin Zhang.
\newblock Equitable coloring of sparse graphs, {arXiv:2411.19801}, 2024.

\bibitem[Mey73]{Meyer1973}
Walter Meyer.
\newblock Equitable coloring.
\newblock {\em Amer. Math. Monthly}, 80(8):920--922, 1973.

\bibitem[Nak12]{NAKPRASIT20121019}
Kittikorn Nakprasit.
\newblock Equitable colorings of planar graphs with maximum degree at least nine.
\newblock {\em Discrete Math.}, 312(5):1019--1024, 2012.

\bibitem[Tuc73]{doi:10.1137/1015072}
Alan Tucker.
\newblock Perfect graphs and an application to optimizing municipal services.
\newblock {\em SIAM Review}, 15(3):585--590, 1973.

\bibitem[YWWL21]{zbMATH07383905}
Wanshun Yang, Yiqiao Wang, Weifan Wang, and Ko-Wei Lih.
\newblock {IC}-planar graphs are 6-choosable.
\newblock {\em SIAM J. Discrete Math.}, 35(3):1729--1745, 2021.

\bibitem[Zha14]{zbMATH06378992}
Xin Zhang.
\newblock Drawing complete multipartite graphs on the plane with restrictions on crossings.
\newblock {\em Acta Math. Sin., Engl. Ser.}, 30(12):2045--2053, 2014.

\bibitem[Zha16]{zbMATH06602931}
Xin Zhang.
\newblock On equitable colorings of sparse graphs.
\newblock {\em Bull. Malays. Math. Sci. Soc. (2)}, 39:s257--s268, 2016.

\bibitem[ZL13]{zbMATH06137073}
Xin Zhang and Guizhen Liu.
\newblock The structure of plane graphs with independent crossings and its applications to coloring problems.
\newblock {\em Cent. Eur. J. Math.}, 11(2):308--321, 2013.

\bibitem[ZWX18]{zbMATH06875705}
Xin Zhang, Hui-juan Wang, and Lan Xu.
\newblock Equitable coloring of three classes of 1-planar graphs.
\newblock {\em Acta Math. Appl. Sin., Engl. Ser.}, 34(2):362--372, 2018.

\bibitem[ZY98]{zbMATH01308943}
Yi~Zhang and Hian-Poh Yap.
\newblock Equitable colorings of planar graphs.
\newblock {\em J. Comb. Math. Comb. Comput.}, 27:97--105, 1998.

\end{thebibliography}

\end{document}